\documentclass{amsart}

\newcommand{\F}{\mathcal F}
\newcommand{\A}{\mathcal A}
\newcommand{\IR}{\mathbb R}
\newcommand{\diam}{\mathrm{diam}}
\newcommand{\conv}{\mathrm{conv}}
\newcommand{\SO}{\mathsf{SO}}
\newcommand{\la}{\langle}
\newcommand{\ra}{\rangle}

\newtheorem{theorem}{Theorem}
\newtheorem{corollary}{Corollary}
\newtheorem{problem}{Problem}
\newtheorem{claim}{Claim}
\newtheorem{example}{Example}
\newtheorem{lemma}{Lemma}
\title[The groups $S^3$ and $\SO(3)$ have no invariant binary $k$-network]{The groups $S^3$ and $\SO(3)$\\ have no invariant binary $k$-network}
\author{Taras Banakh and S\l awomir Turek}
\address{Department of Mathematics, Ivan Franko National University of Lviv, Ukraine}
\address{Instytut Matematyki, Uniwersytet Humanistyczno-Przyrodniczy Jana Kochanowskiego, Kielce, Poland}
\email{t.o.banakh@gmail.com, sturek@ujk.edu.pl}

\begin{document}
\begin{abstract} A family $\mathcal N$ of closed subsets of a topological space $X$ is called a {\em closed $k$-network} if for each open set $U\subset X$ and a compact subset $K\subset U$ there is a finite subfamily $\mathcal F\subset\mathcal N$ with $K\subset\bigcup\F\subset \mathcal N$. A compact space $X$ is called {\em supercompact} if it admits a
closed $k$-network $\mathcal N$ which is {\em binary} in the sense that each linked subfamily $\mathcal L\subset\mathcal N$ is centered. A closed $k$-network $\mathcal N$ in a topological group $G$ is {\em invariant} if $xAy\in\mathcal N$ for each $A\in\mathcal N$ and $x,y\in G$. According to a result of Kubi\'s and Turek \cite{KT}, each compact (abelian) topological group admits an (invariant) binary closed $k$-network. In this paper we prove that the compact topological groups $S^3$ and $\SO(3)$ admit no invariant binary closed $k$-network.
\end{abstract}
\maketitle

\section{Introduction}

In this note we shall discuss the problem of the existence of invariant binary $k$-networks for compact $G$-spaces and compact topological groups.

A family $\A$ of subsets of a set $X$ is called
\begin{itemize}
\item {\em linked} if $A\cap B\ne\emptyset$ for all $A,B\in\A$;
\item {\em centered} if $\cap \F\ne\emptyset$ for any finite subfamily $\F\subset\A$;
\item {\em binary} if each linked subfamily of $\F$ is centered.
\end{itemize}

A family $\A$ of subsets of a topological space $X$ is called a {\em $k$-network} if for any open set $U\subset X$ and a compact subset $K\subset U$ there is a finite subfamily $\F\subset\A$ with $K\subset \cup\F\subset U$, see \cite[\S11]{Gru}. If each set $A\in\A$ of a $k$-network is closed in $X$, then $\A$ will be called a {\em closed $k$-network}.

A compact space $X$ is called {\em supercompact} if $X$ admits a subbase  of the topology such that
each cover of $X$ by elements of the subbase contains a two-element subcover, see \cite{vM}. The following useful characterization of the supercompactness can be derived from Lemma 3.1 of \cite{KT}:

\begin{theorem} A compact Hausdorff space $X$ is supercompact if and only if $X$ admits a binary closed $k$-network.
\end{theorem}

In \cite{Mills} C.Mills proved that each compact topological group $G$ is supercompact, that is  $G$ admits a binary closed $k$-network $\mathcal N$. This result was reproved by W.Kubi\'s and S.Turek \cite{KT} who observed that for an abelian compact topological group $G$ one can construct $\mathcal N$ so that it is {\em left-invariant} in the sense that $xA\in\mathcal N$ for each $A\in\mathcal N$ and $x\in G$. They also asked if such a left-invariant binary $k$-network can be constructed in each compact topological group.

It is natural to consider this problem in the more general context of $G$-spaces.
By a $G$-space we understand a topological space $X$ endowed with a continuous action $\alpha:G\times X\to X$ of a topological group $G$.
A family $\F$ of subsets of a $G$-space $X$ will be called {\em $G$-invariant} if $gF\in\F$ for each $F\in\F$ and each $g\in G$.

A compact $G$-space $X$ will be called {\em $G$-supercompact} if $X$ admits a $G$-invariant binary closed $k$-network.

\begin{problem} Which compact $G$-spaces are $G$-supercompact?
\end{problem}

We shall resolve this problem for the unit sphere $S^n=\{x\in\IR^{n+1}:\|x\|=1\}$ in the Euclidean space $\IR^{n+1}$, endowed with the natural action of the group $\SO(n+1)$ (of orientation preserving linear isometries of $\IR^{n+1}$).

\begin{example}\label{ex}
\begin{enumerate}
\item The $0$-sphere $S^0=\{-1,1\}$ in $\IR$ is $\SO(1)$-supercompact because the family $\F_0=\{\{-1\},\{1\}\}$ of singletons in an $\SO(1)$-invariant binary closed $k$-network for $S^0$.
\item The 1-sphere $S^1$ is $\SO(2)$-supercompact because the family $\F_1$ of all closed connected subsets of diameter less than $\sqrt{3}$ in $S^1$ is an $\SO(2)$-invariant binary closed $k$-network for the circle $S^1$.
\end{enumerate}
\end{example}

It turns out that $S^0$ and $S^1$ are the unique examples of $\SO(n+1)$-supercompact spheres $S^n$.

\begin{theorem}\label{main} The unit sphere $S^n$ in the Euclidean space $\IR^{n+1}$ is $\SO(n+1)$-supercompact if and only if $n\le 1$.
\end{theorem}

This theorem will be proved in Section~\ref{s:p}. Now we shall apply this theorem for finding an example of a compact topological group that admits no invariant binary closed $k$-network.

A family $\F$ of subsets of a group $G$ will be called
\begin{itemize}
\item {\em left-invariant} (resp. {\em right-invariant}) if for each $F\in\F$ and $g\in G$ we get $gF\in\F$ (resp. $Fg\in\F$);
\item {\em invariant} if $\F$ is both left-invariant and right-invariant.
\end{itemize}

It is well-known that the 3-dimensional sphere $S^3$ has the structure of a compact topological group. Namely, $S^3$ is a group with respect to the operation of multiplication of quaternions (with unit norm). It is known \cite[\S4.1]{CS} that for each isometry $f\in \SO(4)$ of $S^3$ there are quaternions $a,b\in S^3$ such that $f(x)=axb$ for all $x\in S^3$. This implies that a family $\F$ of subsets of the group $S^3$ is invariant if and only if it is $\SO(4)$-invariant. Now we see that Theorem~\ref{main} implies:

\begin{corollary}\label{c1} The compact topological group $S^3$ admits no invariant binary closed $k$-network.
\end{corollary}

It is known that the quotient group $S^3/\{-1,1\}$ of $S^3$ by the two-element subgroup $\{-1,1\}$ is isomorphic to the special orthogonal group $\SO(3)$.
Using this fact, we can deduce from Corollary~\ref{c1} the following:

\begin{corollary}\label{c2} The compact topological group $\SO(3)$ admits no invariant binary closed $k$-network.
\end{corollary}

\begin{problem} Has the group $S^3$ or $\SO(3)$  a left-invariant binary $k$-network?
\end{problem}

\begin{problem} Let $G$ be a compact abelian group and $X$ is a compact metrizable $G$-space. Is $X$ $G$-supercompact?
\end{problem}

\begin{problem} Let $G$ be a metrizable (separable) abelian topological group. Has $G$ an invariant binary closed $k$-network?
\end{problem}

\section{Proof of Theorem~\ref{main}}\label{s:p}

First we fix some natation. By $\la\mathbf x,\mathbf y\ra=\sum_{i=1}^nx_iy_i$ we denote the standard inner product of the Euclidean space $\IR^n$. This inner product generates the norm $\|\mathbf x\|=\la \mathbf x,\mathbf x\ra$.
By $S^n=\{x\in\IR^{n+1}:\|x\|=1\}$ we shall denote the unit sphere in  $\IR^{n+1}$.

For an Euclidean space $E=\IR^n$ let $E^*$ be the dual space of $E$, i.e., the space of linear functionals on $E$ endowed with the sup-norm.
By Riesz's Representation Theorem, for each functional $y^*\in E^*$ there is a unique vector $y\in E$ such that $y^*(x)=\la y,x\ra$. So we can identify $E^*$ with $E$.

A {\em convex body} in an Euclidean space $E$ is a convex subset $C\subset E$ with non-empty interior in $E$. By $\partial C$ we denote the boundary of $C$ in $E$.

A functional $y^*\in E^*$ will be called a {\em support functional} to $C$ at a point $c\in\partial C$ if $$y^*(c)=\max x^*(C)>\inf y^*(C).$$ By the Hahn-Banach Theorem, each point $c\in\partial C$ of a convex body $C\subset E$ has a support functional $y^*$ with unit norm.  If such a support functional is unique, then $c$ is called a {\em smooth point} of $\partial C$. It follows from the classical Mazur's Theorem on the differentiablity of continuous convex functions on $E$ that the set of smooth points is dense in $\partial C$.

In an obvious way Theorem~\ref{main} follows from Example~\ref{ex} and the following theorem:

\begin{theorem} For any $n\ge 2$ and any closed subset $A\subset S^n$ of diameter $0<\diam(A)\le 1$ there is an isometry $f\in \SO(n+1)$ such that the family $\{A,f(A),f^2(A)\}$ is linked but not centered.
\end{theorem}

\begin{proof} Let $E=\IR^{n+1}$ and $E^*$ be the dual space to $E$. By $S^*$ we denote the unit sphere in $E^*$.

\begin{lemma}\label{l1} There are distinct points $a_0,a_1\in A$ and a vector $b\in S^*$ such that $\la b,a_0\ra=0=\max_{a\in A}\la b,a\ra$ and $\la b,a_1\ra>-\frac12\|a_1-a_0\|$.
\end{lemma}

\begin{proof} The lemma trivially holds if there are a vector $b\in S^*$ and two distinct points $a_0,a_1\in A$ such that  $\la b,a_0\ra=\la b,a_1\ra =\max_{a\in A}\la b,a\ra=0$.

So, assume that no such vectors $b$, $a_0,a_1$ exist.
Let $L_A$ be the linear hull of the set $A$ and $C\subset L_A$  be the closed convex hull of the set $A\cup\{0\}$ in $L_A$. Since the set $A\subset S^n$ contains more than one point, the linear space $L_A$ has dimension $\dim L_A\ge 2$. It is clear that $C$ is a convex body in $L_A$.
By Mazur's Theorem, the set of smooth points is a dense in the boundary $\partial C$. Consequently, there is a smooth point $c\in \partial C$ such that $0<\|c\|<1$. Let $b^*\in L_A^*$ be the unique norm one support functional to $C$ at the point $c$. Let $a_0=\dfrac{c}{\|c\|}$ and observe that $a_0\in \conv(A)\subset C$. Since $b^*$ is a support functional at $c$, we get
$b^*(c)=\max b^*(C)\ge b^*(0)=0$. We claim that $b^*(c)=0$. The strict inequality $b^*(c)>0$ would imply $b^*(c)=0=\max b^*(C)\ge b^*(a_0)=\dfrac{b^*(c)}{\|c\|}$ and $\|c\|\ge 1$, which contradicts the choice of $c$.

Let us show that the point $a_0=c/\|c\|$ belongs to the set $A$.
Since $c\in \conv(A\cup\{0\})\setminus\{0\}$ by the Caratheodory Theorem, there are pairwise distinct points $a_1,\dots,a_k\in A$ and positive real numbers
$\lambda_1,\dots,\lambda_k$ such that $\sum_{i=1}^k\lambda_i\le 1$ and $c=\sum_{i=1}^k\lambda_ia_i$. This equality and $b^*(c)=0=\max b^*(A)$ imply that $b^*(a_i)=0$ for all $1\le i\le k$. Now our assumption guarantees that $k=1$ (otherwise, $a_1$ and $a_2$ are two distinct points with $b^*(a_1)=b^*(a_2)=\max b^*(A)=0$, which is forbidden by our assumption). Therefore, $c=\lambda_1a_1$ and hence $a_0=c/\|c\|=c/\lambda_1=a_1\in A$.
\smallskip

Let $c^*\in L_A^*$ be any functional with unit norm such that $c^*(a_0)=0$ and $0<\|b^*-c^*\|\le \frac12$. Since the functional $c^*\ne b^*$ is not support at the point $c$, there is a point $a_1\in A$ such that $c^*(a_1)>0$.

Observe that
$$
\begin{aligned}
b^*(a_1)&= b^*(a_1-a_0)\ge c^*(a_1-a_0)-\|c^*-b^*\|\cdot\|a_1-a_0\|=\\
&=c^*(a_1)-\frac12\|a_1-a_0\|>-\frac12\|a_1-a_0\|
\end{aligned}.$$

By Riesz's Representation Theorem, the functional $b^*$ can be identified with a unique vector $b\in L_A\subset E$ such that $b^*(x)=\la b,x\ra$ for all $x\in L$. The vector $b$ and the points $a_0,a_1$ have the properties required in Lemma~\ref{l1}.
\end{proof}

Let $L$ be the 3-dimensional linear subspace of $E$ generated by the vectors $b,a_0,a_1$ (from Lemma~\ref{l1}) and let $L^\perp\subset E$ be its orthogonal complement. Then the space $E$ decomposes into the direct sum $L\oplus L^\perp$.

Find a (unique) point $a_2$ in the 2-sphere $L\cap S^n$ such that $\|a_2-a_0\|=\|a_2-a_1\|=\|a_1-a_0\|$ and $\la b,a_2\ra>0$. Let  $c=\frac13(a_0+a_1+a_2)$ be the center of the equilateral triangle $\triangle a_0,a_1,a_2$. It follows from $\la b,a_0\ra=0$ and $0\ge \la b,a_1\ra>-\frac12\|a_0-a_1\|$ that $\la b,a_2\ra>\frac12\|a_0-a_1\|$.
Consequently,
\begin{itemize}
\item[$(*)$] $\la b,c\ra=\frac13(\la b,a_1\ra+\la b,a_2\ra)>0$.
\end{itemize}

\begin{claim}\label{cl2} $\la c,a\ra>0$ for each $a\in A$.
\end{claim}

\begin{proof} Observe that $\la a_2,a_0\ra=\frac12(\|a_2\|^2+\|a_0\|^2-\|a_2-a_0\|^2)\ge\frac12(1+1-1)=\frac12$ and then $\|-a_2-a_0\|^2=\|a_2\|^2+\|a_0\|^2+2\la a_2,a_0\ra\ge 3$, which implies that $-a_2\notin A$ because $\diam(A)\le 1$. Then for each $a\in A$ we get $a_2\ne -a$ and hence $\la a_2,a\ra>-\|a_2\|\cdot\|a\|=-1$.

On the other hand, for $i\in\{0,1\}$ we get $$\la a_i,a\ra=\frac12(\|a_i\|^2+\|a\|^2-\|a_0-a\|^2)\ge \frac12(1+1-1)=\frac12.$$
Then $$\la c,a\ra=\tfrac13\la a_0+a_1+a_2,a\ra=\tfrac13(\la a_0,a\ra+\la a_1,a\ra+\la a_2,a\ra)>\tfrac13(\tfrac12+\tfrac12-1)=0.$$
\end{proof}

Let $R:L\to L$ be the rotation of the 3-dimensional Euclidean space $L$ around the axis $\IR c$ on the angle $2\pi/3$ such that
$R(a_0)=a_1$, $R(a_1)=a_2$ and $R(a_2)=a_0$.
Extend $R$ to an isometry $f\in \SO(n+1)$ of $E=L\oplus L^\perp$ letting $f(x+y)=R(x)+y$ for $(x,y)\in L\times L^\perp$.
It remains to prove:

\begin{claim} The system $\mathcal L=\{A,f(A),f^2(A)\}$ is linked but not centered.
\end{claim}

\begin{proof} The linkedness of the system $\mathcal L$ follows from the inclusion $\{a_0,a_1\}\subset A$ and the linkedness of the system
$$\big\{\{a_0,a_1\},\{a_1,a_2\},\{a_2,a_0\}\big\}=\big\{\{a_0,a_1\},f(\{a_0,a_1\}),f^2(\{a_0,a_1\})\big\}.$$

To see that $\mathcal L$ is not centered, consider the half-spaces
$H_b=\{x\in E:\la b,x\ra\le 0\}$ and $H_c=\{x\in E:\la c,x\ra>0\}$. The choice of the vectors $b,a_0,a_1,a_2$ guarantees that $a_0,a_1\in A\subset H_b$ but $a_2,c\notin H_b$. By Claim~\ref{cl2}, $A\subset H_c$.

Let $H_c^L=H_c\cap L$ and $H_b^L=H_b\cap L$.
The inclusions $b,c\in L$ imply that $H_b=H_b^L\oplus L^\perp$ and $H_c=H_c^L\oplus L^\perp$.

It follows that $R(H_c^L)=H_c^L$ and hence $f(H_c)=H_c$. Observe that $$A\cap f(A)\cap f^2(A)\subset H_c\cap H_b\cap f(H_b)\cap f^2(H_b)=(H_c^L\cap H_b^L\cap R(H_b^L)\cap R^2(H_b^L))\oplus L^\perp.$$

Now to see that $A\cap f(A)\cap f^2(A)=\emptyset$ it suffices to prove that the intersection $H^L=H_c^L\cap H_b^L\cap R(H_b^L)\cap R^2(H_b^L)$ is empty. Assuming that this intersection contains some point $h$, we conclude that it contains its rotations $R(h)$ and $R^2(h)$ and also the center $c_h=\frac13(h+R(h)+R^2(h))$ of the equilateral triangle $\{h,R(h),R^2(h)\}$ (by the convexity of $H^L$). The center $c_h$ lies on the axis $\IR\cdot c$ of the rotation $R$. Taking into account that $c_h\in H_c$, we conclude that $\la c,c_h\ra>0$ and hence $c\in (0,+\infty)\cdot c_h\subset H_b$, which contradicts the inequality $(*)$.
\end{proof}
\end{proof}

\end{document}